\documentclass[12pt]{amsart}
\usepackage{amssymb,amsmath}
\usepackage{latexsym}
\usepackage{amsthm,amsfonts,amssymb,mathrsfs}
\usepackage{rotating}
\usepackage[leqno]{amsmath}
\usepackage{xspace}
\usepackage[all]{xy}
\usepackage{longtable}
\usepackage{amstext}
\usepackage{amscd}
\usepackage{enumerate}
\usepackage{setspace}
\usepackage{mathrsfs}
\usepackage{verbatim}
\usepackage{tikz}
\usetikzlibrary{matrix,calc}
\headsep=1cm \numberwithin{equation}{section}
\makeatletter \@namedef{subjclassname@2020}{\textup{2020} Mathematics Subject Classification}
\makeatother

\newtheorem{theorem}{Theorem}[section]
\newtheorem{lemma}[theorem]{Lemma}

\newtheorem{corollary}[theorem]{Corollary}

\theoremstyle{definition}
\newtheorem{definition}[theorem]{Definition}
\theoremstyle{remark}
\newtheorem{remark}[theorem]{Remark}

\newcommand{\hooklongrightarrow}{\lhook\joinrel\longrightarrow}

\newcommand{\Ass}{\operatorname{Ass}}
\newcommand{\im}{\operatorname{im}}

\newcommand{\Assh}{\operatorname{Assh}}
\newcommand{\Spec}{\operatorname{Spec}}

\newcommand{\rank}{\operatorname{rank}}
\newcommand{\id}{\operatorname{id}}

\newcommand{\pd}{\operatorname{pd}}

\newcommand{\Ext}{\operatorname{Ext}}
\newcommand{\Supp}{\operatorname{Supp}}

\newcommand{\Tor}{\operatorname{Tor}}
\newcommand{\Hom}{\operatorname{Hom}}

\newcommand{\Att}{\operatorname{Att}}
\newcommand{\Ann}{\operatorname{Ann}}
\newcommand{\amp}{\operatorname{amp}}
\newcommand{\cmd}{\operatorname{cmd}}

\newcommand{\depth}{\operatorname{depth}}

\newcommand{\lo}{\longrightarrow}
\newcommand{\fm}{\mathfrak{m}}
\newcommand{\fp}{\mathfrak{p}}
\newcommand{\fq}{\mathfrak{q}}

\DeclareMathSymbol{\perp}{\mathrel}{symbols}{"3F}
\begin{document}

\author[K. Abolfath Beigi, K. Divaani-Aazar and M. Tousi]{Kosar Abolfath Beigi, Kamran Divaani-Aazar
and Massoud Tousi}

\title[Characterizing certain semidualizing ...]{Characterizing certain semidualizing complexes via their
Betti and Bass numbers}

\address{K. Abolfath Beigi, Department of Mathematics, Faculty of Mathematical Sciences, Alzahra
University, Tehran, Iran.}
\email{kosarabolfath@gmail.com}

\address{K. Divaani-Aazar, Department of Mathematics, Faculty of Mathematical Sciences, Alzahra
University, Tehran, Iran-and-School of Mathematics, Institute for Research in Fundamental Sciences (IPM),
P.O. Box 19395-5746, Tehran, Iran.}
\email{kdivaani@ipm.ir}

\address{M. Tousi, Department of Mathematics, Faculty of Mathematical Sciences, Shahid Beheshti University,
Tehran, Iran.}
\email{mtousi@ipm.ir}



\begin{abstract} It is known that the numerical invariants Betti numbers and Bass numbers are worthwhile tools
for decoding a large amount of information about modules over commutative rings. We highlight this fact, further,
by establishing some criteria for certain semidualizing complexes via their Betti and Bass numbers.
Two distinguished types of semidualizing complexes are the shifts of the underlying rings and dualizing complexes.
Let $C$ be a semidualizing complex for an analytically irreducible local ring $R$ and set $n:=\sup C$ and
$d:=\dim_RC$. We show that $C$ is quasi-isomorphic to a shift of $R$ if and only if the $n$th Betti number of $C$
is one. Also, we show that $C$ is a dualizing complex for $R$ if and only if the $d$th Bass number of $C$ is one.
\end{abstract}

\maketitle

\section{Introduction}

Grothendieck and Hartshorne introduced dualizing complexes; see \cite{H}. Dualizing complexes and the rich duality
theory induced by them were appeared in many contexts in algebraic geometry and commutative algebra. Semidualizing
modules are also proven to be a venerable tool in commutative algebra. To unify the study of dualizing complexes and
semidualizing modules, Christensen \cite{C1} introduced the notion of semidualizing complexes. Immediate examples of
semidualizing complexes are the shifts of the underlying rings and dualizing complexes. These are precisely semidualizing
complexes with finite projective dimension and semidualizing complexes with finite injective dimension. It is known
that a homologically bounded complex $X$ with finitely generated homology modules is a dualizing complex for a local
ring $R$ if and only if there exists an integer $d$ such that the $d$th Bass number of $X$ is one and the other Bass
numbers of $X$ are zero; see \cite[Chapter V, Proposition 3.4]{H}. We establish some criteria for semidualizing complexes
that are either the shifts of the underlying ring or dualizing, via their Betti and Bass numbers.

Throughout, $(R,\fm,k)$ is a commutative Noetherian local ring with non-zero identity. Since the time Hilbert proved his
celebrated Syzygy theorem, the importance of free resolutions becomes clear. The numerical invariants Betti numbers of a
finitely generated $R$-module $M$ are defined by the minimal free resolution of $M$. By definition, the $i$th Betti number
of $M$, denoted $\beta^R_i(M)$, is the minimal number of generators of the $i$th syzygy of $M$ in its minimal free resolution.
By knowing the Betti numbers of an $R$-module, we can decode a lot of valuable information about it.

The $i$th Bass number of an $R$-module $M$, denoted $\mu^i_R(M)$, is defined by using the minimal injective resolution of
$M$. Similar to the Betti numbers, the Bass numbers also reveal a lot of information about the module. The type of an $R$-module
$M$ is defined to be $\mu^{\dim_RM}_R(M)$. It is known that local rings with small type have nice properties; see e.g.
\cite{B, V, F1, R, CHM, K, Ao, Le1, Le2, CSV}. By the work of Bass \cite{B}, every Cohen-Macaulay local ring of type one is
Gorenstein. (In \cite{CSV}, this is improved to the statement that if $R$ is a Cohen-Macaulay local ring and $\mu^n_R(R)=1$
for some $n\geq 0$, then $R$ is Gorenstein of dimension $n$.) Vasconcelos \cite[4, p. 53]{V} conjectured that every local ring
of type one is Gorenstein. Foxby \cite{F1} solved this conjecture for equicharacteristic local rings. The general case was
answered by Roberts \cite{R}.

In this article, we prove that if $X$ is a Cohen-Macaulay complex with $\beta^R_{\sup X}(X)=1$, then $X\simeq
\Sigma^{\sup X} R/\Ann_RX$; see Theorem \ref{29}(a). As an application, we show that if $M$ is a finitely generated
$R$-module of type one, then $M$ is the dualizing module for the ring $R/\text{Ann}_RM$; see Corollary \ref{210}.
This gives an affirmative answer to a conjecture raised by Foxby \cite[Conjecture B]{F1}, which is already proved;
see e.g. \cite[Corollary 2.7]{Le1}. We may improve Theorem \ref{29}(a), if we consider a semidualizing complex $C$
with $\beta^R_{\sup C}(C)\leq 2$. We do this in Theorems \ref{29}(b), \ref{212} and \ref{216}. These results
assert that:

\begin{theorem}\label{11} Let $(R,\fm,k)$ be a local ring and $C$ a semidualizing complex for $R$.
\begin{enumerate}
\item[(a)] If $\beta^R_{\sup C}(C)=1$ and $\Ass R$ is a singleton, then $C\simeq \Sigma^{\sup C}R$.
\item[(b)] If $\beta^R_{\sup C}(C)=2$ and $R$ is a domain, then $\amp C=0$.
\item[(c)] If $\beta^R_{n}(C)=1$ for some integer $n$ and either $R$ is Cohen-Macaulay, or $C$ is a module with a rank,
then $C\simeq \Sigma^nR$ and $n=\sup C$.
\end{enumerate}
\end{theorem}

Dually, concerning semidualizing complexes with small type, in Corollaries \ref{213} and \ref{217}, we prove
that:

\begin{theorem}\label{12} Let $(R,\fm,k)$ be a local ring and $C$ a semidualizing complex for $R$.
\begin{enumerate}
\item[(a)] If $C$ is a module of type one, then $C$ is dualizing.
\item[(b)] If $C$ has type one, then $C$ is dualizing provided either $\Ass R$ is a singleton and $R$ admits a dualizing
complex, or $\Ass \widehat{R}$ is a singleton.
\item[(c)] If $C$ has type two and $R$ is a domain admitting a dualizing complex, then $C$ is Cohen-Macaulay.
\item[(d)] If $\mu_R^{n}(C)=1$ for some integer $n$, $C$ is Cohen-Macaulay and $R$ is either Cohen-Macaulay or analytically
irreducible, then $C$ is dualizing and $n=\dim_RC$.
\end{enumerate}
\end{theorem}

Each of parts (a) and (b) extends the above-mentioned result of Roberts. Also, part (c) extends the main result of \cite{CHM}.
Finally, part (d) generalizes \cite[Characterization 1.6]{CSV}.

It is known that every dualizing complex has type one. Also, it is easy to check that if a complex $C$ is quasi-isomorphic to a
shift of $R$, then $\beta^R_{\sup C}(C)=1$. Thus, we immediately exploit Theorem \ref{11}(a) and Theorem \ref{12}(b) to obtain the
following corollary:

\begin{corollary}\label{13} Let $(R,\fm,k)$ be an analytically irreducible local ring and $C$ a semidualizing complex for $R$. Then
\begin{enumerate}
\item[(a)] $C$ is quasi-isomorphic to a shift of $R$ if and only if $\beta^R_{\sup C}(C)=1$.
\item[(b)] $C$ is a dualizing complex for $R$ if and only if $\mu^{\dim_RC}_R(C)=1$.
\end{enumerate}
\end{corollary}

\section{The Results}

In what follows, $\mathcal{D}(R)$ denotes the derived category of the category of $R$-modules. Also, $\mathcal{D}^f_{\Box}(R)$
denotes the full subcategory of homologically bounded complexes with finitely generated homology modules. The isomorphisms
in $\mathcal{D}(R)$ are marked by the symbol $\simeq$. For a complex $X\in \mathcal{D}(R)$, its supremum and infimum are defined,
respectively, by $\sup X:=\sup \{i\in \mathbb{Z}\mid \text{H}_i(X)\neq 0\}$ and $\inf X:=\inf \{i\in \mathbb{Z}\mid
\text{H}_i(X) \neq 0\}$, with the usual convention that $\sup \emptyset=-\infty$ and $\inf \emptyset=\infty$. Also, amplitude of
$X$ is defined by $\amp X:=\sup X-\inf X$. In what follows, we will use, frequently, the endofunctors $-\otimes_R^{{\bf L}}-$,
${\bf R}\Hom_R(-,-)$ and $\Sigma^n(-), n\in\mathbb{Z},$ on the category $\mathcal{D}(R)$. For the definitions and basic properties of
these functors, we refer the reader to \cite[Appendix]{C2}.

Theorem \ref{29} is the first main result of this paper. To prove it, we need Lemmas \ref{25}, \ref{27} and \ref{28}. For proving
Lemma \ref{27}, we require to establish Lemmas \ref{21}, \ref{22} and \ref{26}. Lemma \ref{21} improves, slightly,  the celebrated
New Intersection theorem. Our proofs of Lemmas \ref{21}, \ref{22} are slight modifications of the proofs of \cite[Theorem 1.2]{F1}
and \cite[Lemma 3.2]{F1}. Also, Remark \ref{24} serves as a powerful tool throughout the paper.

\begin{lemma}\label{21} Let $(R,\fm,k)$ be a local ring. Let $$F= 0\lo F_s\lo F_{s-1}\lo \cdots \lo F_0\lo 0$$ be a non-exact
complex of finitely generated free $R$-modules and $t$ be a non-negative integer such that $\dim_R(\text{H}_i(F))\leq i+t$ for
all $0\leq i\leq s$. Then $\dim R\leq s+t$.
\end{lemma}

\begin{proof} Without loss of generality, we may and do assume that $R$ is complete. Let $\fp$ be a prime ideal of $R$ with
$\dim R/\fp=\dim R$ and set $\widetilde{F}:=F\otimes_R R/\fp$. Then $\widetilde{F}$ is a bounded complex of finitely generated
free $R/\fp$-modules.  By virtue of \cite[Corollary A.4.16]{C2}, we have $\inf \widetilde{F}=\inf F$, and so the complex $\widetilde{F}$ is
non-exact.

Next, we show that $\dim_{R/\fp}(\text{H}_i(\widetilde{F}))\leq i+t$ for all $0\leq i\leq s$. Suppose that the contrary holds.
Then, we may choose $\ell$ to be the least integer such that $\dim_{R/\fp}(\text{H}_{\ell}(\widetilde{F}))>\ell+t$. So, there is
a prime ideal $\fq/\fp\in \Supp_{R/\fp}(\text{H}_{\ell}(\widetilde{F}))$ such that $\dim (R/\fp)/(\fq/\fp)>\ell+t$. From the
choice of $\ell$, it turns out that $\inf((\widetilde{F})_{\fq/\fp})=\ell$. But,
$$\begin{array}{ll}
(\widetilde{F})_{\fq/\fp}&\cong (F\otimes_RR/\fp)\otimes_{R/\fp}(R/\fp)_{\fq/\fp}\\
&\cong F\otimes_R R_{\fq}/\fp R_{\fq} \\
&\cong (F\otimes_RR_{\fq})\otimes_{R_{\fq}}R_{\fq}/\fp R_{\fq} \\
&\cong F_{\fq}\otimes_{R_{\fq}}R_{\fq}/\fp R_{\fq},\\
\end{array}$$
and so applying \cite[Corollary A.4.16]{C2} again, yields that $\inf F_{\fq}=\ell$. Hence $\fq\in \Supp_R(\text{H}_{\ell}(F))$,
and so $\dim_R(\text{H}_{\ell}(F))>\ell+t$. We arrive at a contradiction.  Therefore, in the rest of proof, we assume that $R$
is a catenary local domain.

We proceed by induction on $d:=\dim R$. Obviously, the claim holds for $d=0$. Suppose that $d>0$ and the claim
holds for $d-1$. If $\dim_R(\text{H}_i(F))\leq 0$ for all $0\leq i\leq s$, then by the New Intersection theorem, we deduce
that $\dim R\leq s$. So, we may assume that $\dim_R(\text{H}_j(F))> 0$ for some $0\leq j\leq s$. Then, we can take a prime ideal
$\fp\in \Supp_R(\text{H}_j(F))$ such that $\dim R/\fp=1$. Now, $F_{\fp}$ is a non-exact complex of finitely generated free
$R_{\fp}$-modules and $$\dim_{R_{\fp}}(\text{H}_i(F_{\fp}))\leq \dim_R(\text{H}_i(F))-1\leq i+(t-1)$$ for all $0\leq i\leq s$.
Since $\dim R_{\fp}=\dim R-1$, the induction hypothesis implies that $\dim R_{\fp}\leq s+(t-1)$. Hence, $\dim R\leq s+t$.
\end{proof}

To present the next result, we have to fix some notation.  Let $t$ be a non-negative integer and $$F= \cdots \lo F_i\lo F_{i-1}
\lo \cdots \lo F_{t+1}\lo F_t\lo 0$$ be a complex of finitely generated free $R$-modules. For each $i\geq t$, set $$\gamma_i(F):
=r_i-r_{i-1}+\cdots +(-1)^{i-t}r_t,$$ where $r_j$ is the rank of $F_j$ for each $j\geq t$.

\begin{lemma}\label{22} Let $(R,\fm,k)$ be a local ring and $t\leq \dim R$ a non-negative integer. Let $$F=\cdots\lo F_i\lo F_{i-1}
\lo \cdots \lo F_t\lo 0$$ be a complex of finitely generated free $R$-modules such that $\dim_R(\text{H}_i(F))\leq i$ for all
$t\leq i\leq \dim R$. Then for each integer $t\leq n\leq \dim R$, we have:
\begin{itemize}
\item[(a)] $\gamma_n(F)\geq 0$.
\item[(b)] $\gamma_n(F)>0$ if $\text{H}_t(F)\neq 0$ and either $\dim_R(\text{H}_n(F))=n$ or $n<\dim R$.
\end{itemize}
\end{lemma}

\begin{proof} Let the complex $$G=0\lo F_n\lo F_{n-1}\lo \cdots \lo F_t\lo 0$$ be the truncation of $F$ at the spot $n$ and let
$Z_n$ denote the kernel of the differential map $F_n\lo F_{n-1}$. Then, $\text{H}_i(G)=\text{H}_i(F)$ for all $i<n$ and
$\text{H}_n(G)=Z_n$. If $n=t$, then $\gamma_n(F)$ is the rank of the free $R$-module $F_t$, and so both assertions are immediate
in this case. So, in the rest of the proof, we assume that $n>t$.

(a) Let $\fp$ be a prime ideal of $R$ with $\dim R/\fp\geq n$. As $\dim_R(\text{H}_i(F))\leq i$ for all $t\leq i\leq n-1$, it
follows that the complex $$0\lo (Z_n)_{\fp}\lo (F_n)_{\fp}\lo (F_{n-1})_{\fp}\lo \cdots \lo (F_t)_{\fp}\lo 0$$ is exact. Hence,
$(Z_n)_{\fp}$ is a free $R_{\fp}$-module of rank $\gamma_n(F)$. So, $\gamma_n(F)\geq 0$.

(b) Assume that $\text{H}_t(F)\neq 0$. Suppose that $\dim_R Z_n<n$. Then, applying Lemma \ref{21} to the complex $G$, we get
$\dim R\leq n$. Also, we have $$\dim_R(\text{H}_n(F))\leq \dim_R Z_n<n.$$ Thus, if either $\dim_R(\text{H}_n(F))=n$ or
$n<\dim R$, then $\dim_R Z_n\geq n$. Let $\fp\in \Supp_R Z_n$ be such that $\dim R/\fp\geq n$. Then, as we saw in the proof of
(a), the non-zero $R_{\fp}$-module $(Z_n)_{\fp}$ is free of rank $\gamma_n(F)$, and so $\gamma_n(F)>0$.
\end{proof}

Let $(R,\fm,k)$ be a local ring. For an $R$-complex $X$ and an integer $n$, the $n$th Betti number of $X$, $\beta^R_n(X)$, is
defined as the rank of the $k$-vector space $\Tor^R_n(k,X)(:=\text{H}_{n}(k\otimes_R^{{\bf L}}X))$. Dually, the $n$th Bass number
of $X$, $\mu^n_R(X)$, is defined as the rank of the $k$-vector space $\Ext_R^n(k,X)(:=\text{H}_{-n}({\bf R}\Hom_R(k,X)))$. Also,
dimension, depth and Cohen-Macaulay defect of an $R$-complex $X$ are, respectively, defined as follows: $$\dim_RX:=\sup\{\dim_R(
\text{H}_i(X))-i \mid \ i\in\mathbb{Z}\},$$  $$\depth_RX:=\inf\{i\in \mathbb{Z} \mid \Ext_R^i(k,X)\neq 0 \}, \  \  \text{and}$$
$$\cmd_RX:=\dim_RX-\depth_RX.$$ For an $R$-complex $X$, it is known that $\dim_RX=\sup \{\dim R/\fp-\inf X_{\fp} \mid \fp\in
\Spec R\}.$ We set $$\Assh_RX:=\{\fp\in \Spec R \mid \dim R/\fp-\inf X_{\fp}=\dim_RX\}.$$ Finally, the type of an $R$-complex
$X$ is defined to be $\mu^{\dim_RX}_R(X)$.

\begin{definition}\label{23} Let $(R,\fm,k)$ be a local ring. A non-exact complex $X\in \mathcal{D}^f_{\square}(R)$ is
called Cohen-Macaulay if $\depth_RX=\dim_RX.$
\end{definition}

Let $(R,\fm,k)$ be a local ring. A \emph{semidualizing complex} for $R$ is a complex $C\in \mathcal{D}_{\Box}^f(R)$ such that
the homothety morphism $R\longrightarrow {\bf R}\Hom_R(C,C)$ is an isomorphism in $\mathcal{D}(R)$. If, furthermore, $C$ has
finite injective dimension, then it is called a \emph{dualizing complex}. It is obvious that any shift of a (semi)dualizing
complex is again a (semi)dualizing complex. A dualizing complex $D$  satisfying $\sup D=\dim R$, is called a \emph{normalized
dualizing complex}.  Given a dualizing complex $D$, it is immediate that $\Sigma^{\dim R-\sup D}D$ is a normalized dualizing
complex.  For a normalized dualizing complex $D$, it is known that $\mu^i_R(D)=0$ for all $i\neq 0$ and $\mu^0_R(D)=1$; in
particular $D$ has depth zero.

An $R$-module which is a (semi)dualizing complex for $R$ is said to be a \emph{(semi)dualizing module}. Immediate examples of
semidualizing modules are free $R$-modules of rank one. It is known that the ring $R$ possesses a dualizing complex if and only
if it is homomorphic image of a Gorenstein local ring. In particular, in view of the Cohen Structure theorem, every complete
local ring admits a (normalized) dualizing complex. Also, by virtue of \cite[Theorem 3.3.6]{BH}, we can see that the ring $R$
admits a dualizing module if and only if it is Cohen-Macaulay and it is homomorphic image of a Gorenstein local ring.

\begin{remark}\label{24} Let $(R,\fm,k)$ be a local ring and $D$ a normalized dualizing complex for $R$. There is a dagger
duality functor $$(-)^{\dag}:={\bf R}\Hom_R(-,D):\mathcal{D}_{\Box}^f(R)\lo \mathcal{D}_{\Box}^f(R),$$ such that for every
$X\in \mathcal{D}_{\Box}^f(R)$, the following assertions hold:
\begin{itemize}
\item[(a)] $(X^{\dag})^{\dag}\simeq X$.
\item[(b)] $\dim_RX=\sup X^{\dag}$ and $\depth_RX=\inf X^{\dag}$.
\item[(c)] $X$ is a Cohen-Macaulay complex if and only if $\amp X^{\dag}=0$.
\item[(d)] $\mu^n_R(X)=\beta^R_n(X^{\dag})$ for all integers $n$.
\item[(e)] $\id_RX=\pd_RX^{\dag}$.
\item[(f)] $X$ is a semidualizing complex for $R$ if and only if $X^{\dag}$ is so.
\end{itemize}
\end{remark}

\begin{proof} For parts (a) and (b) see e.g. \cite[Theorem A.8.5]{C2}. For part (b), note that $\depth_RD=0$.

(c) is obvious by (b).

(d) As $D$ is a normalized dualizing complex for $R$, we have $\mu^i_R(D)=0$ for all $i\neq 0$ and $\mu^0_R(D)=1$. Thus, for
every integer $n$, \cite[Theorem 4.1(a)]{F2} yields that $$\mu^n_R(X)=\mu^n_R({\bf R}\Hom_R(X^{\dag},D))=\sum \limits_{i\in
\mathbb{Z}}\beta^R_i(X^{\dag}) \mu^{n-i}_R(D)=\beta^R_n(X^{\dag}).$$

(e) For each complex $Y\in \mathcal{D}_{\Box}^f(R)$, by virtue of \cite[Corollary 2.10.F and Proposition 5.5]{AF}, it turns out
that $\pd_RY=\sup (k\otimes_R^{\bf L}Y)$ and $\id_RY=-\inf ({\bf R}\Hom_R(k,Y))$, and so $$\pd_RY=\sup \{i\in \mathbb{Z}
\mid \beta^R_i(Y)\neq 0\}, \  \  \text {and}$$ $$\id_RY=\sup \{i\in \mathbb{Z} \mid \mu^i_R(Y)\neq 0\}.$$ Therefore, (d)
implies that $\id_RX=\pd_RX^{\dag}$.

(f) See \cite[Corollary 2.12]{C1}.
\end{proof}

In the statement and proof of the next result, we use the notions of weak annihilator of a complex and attached prime ideals
of an Artinian module. We recall these notions. For an Artinian $R$-module $A$,  the set of attached prime ideals of $A$,
$\Att_RA$, is the set of all prime ideals $\fp$ of $R$ such that $\fp=\Ann_RL$ for some quotient $L$ of $A$; see e.g.
\cite[Chapter 7]{BS} for more details. Following \cite{Ap}, for a complex $X\in\mathcal{D}(R)$, the weak annihilator of $X$
is defined as $\Ann_RX:=\bigcap \limits_{i\in \mathbb{Z}} \Ann_R(\text{H}_i(X))$.

\begin{lemma}\label{25} Let $(R,\fm,k)$ be a local ring and $X\in \mathcal{D}_{\Box}^f(R)$ a non-exact complex. Let
$n:=\sup X$ and $d:=\dim_RX$.
\begin{itemize}
\item[(a)] $\dim_R(\text{H}_i(X))\leq i+d$ for all $i\leq n$.
\item[(b)] If $X$ is Cohen-Macaulay, then $\dim_R(\text{H}_n(X))=n+d$.
\item[(c)] If $R$ possesses a normalized dualizing complex $D$ and $\amp X=0=\amp X^{\dag}$, then $\Ann_RX=\Ann_RX^{\dag}$.
\end{itemize}
\end{lemma}

\begin{proof} (a) It is obvious, because $$d=\dim_RX=\sup\{\dim_R(\text{H}_i(X))-i\mid\ i\in\mathbb{Z}\}.$$

(b) Without loss of generality, we may and do assume that $R$ is complete. Then $R$ possesses a normalized dualizing complex
$D$. By virtue of Remark \ref{24}, we have $\amp X^{\dag}=0$, $\sup X^{\dag}=d$ and $\dim_RX^{\dag}=n$.
So, $X^{\dag}\simeq \Sigma^{d}\text{H}_d(X^{\dag})$. Set $N:=\text{H}_d(X^{\dag})$. As $\dim_RX^{\dag}=n$, it follows that
$\dim_RN=n+d$.  Thus, in view of Remark \ref{24}(a) and the Local Duality theorem \cite[Chapter V, Theorem 6.2]{H}, we have the
following display of isomorphisms:
$$\begin{array}{ll}
\text{H}_{\fm}^{\dim_RN}\left(N\right)&\cong \Hom_R\left(\text{H}_{n+d}\left(N^{\dag}\right) ,\text{E}_R(R/\fm)\right)\\
&\cong \Hom_R\left(\text{H}_n\left(\Sigma^{-d}N^{\dag}\right),\text{E}_R(R/\fm)\right)\\
&\cong \Hom_R\left(\text{H}_n\left(\left(\Sigma^{d}N\right)^{\dag}\right),\text{E}_R(R/\fm)\right)\\
&\cong \Hom_R\left(\text{H}_n\left(\left(X^{\dag}\right)^{\dag}\right),\text{E}_R(R/\fm)\right)\\
&\cong \Hom_R\left(\text{H}_n(X),\text{E}_R(R/\fm)\right).
\end{array}$$
Now, \cite[Theorem 7.3.2 and Exercise 7.2.10(iv)]{BS} yields that $$\Ass_R(\text{H}_n(X))=\Att_R(\text{H}_{\fm}^{\dim_RN}
\left(N\right))=\{\fp\in \Ass_RN\mid\ \dim R/\fp=\dim_RN \},$$ and so $\dim_R(\text{H}_n(X))=n+d$.

(c) First of all note that for an $R$-module $L$, an $R$-complex $Y$ and an integer $i$, we may easily verify that
$\Ann_RL\subseteq \Ann_R(\text{H}_i(L^{\dag}))$ and $\Ann_RY=\Ann_R(\Sigma^{i}Y)$. Set $M:=\text{H}_n(X)$ and
$N:=\text{H}_d(X^{\dag})$. Then $X\simeq \Sigma^{n}M$ and $X^{\dag}\simeq \Sigma^{d}N$. Thus, the claim is equivalent to
$\Ann_RM=\Ann_RN$. Now, we have $M^{\dag}\simeq \Sigma^{n+d}N$ and $N^{\dag}\simeq \Sigma^{n+d}M$, and so
$$\begin{array}{ll}
\Ann_RM&\subseteq \Ann_R(M^{\dag})\\
&=\Ann_RN\\
&\subseteq \Ann_R(N^{\dag})\\
&=\Ann_RM.
\end{array}$$
\end{proof}

\begin{lemma}\label{26} Let $(R,\fm,k)$ be a local ring. Let $X\in \mathcal{D}^f_{\Box}(R)$ be a non-exact complex such that
$\beta^R_{\sup X}(X)=1$ and $\amp X=0$. Then $X\simeq \Sigma^{\sup X}R/\Ann_RX$. Furthermore:
\begin{itemize}
\item[(a)] If  $X$ is Cohen-Macaulay, then the local ring $R/\Ann_RX$ is also Cohen-Macaulay.
\item[(b)] If $\Ann_RX=0$, then $\pd_RX<\infty$.
\end{itemize}
\end{lemma}

\begin{proof} Set $R_X:=R/\Ann_RX$, $n:=\sup X$ and $M:=\text{H}_n(X)$. Then $X\simeq \Sigma^n M$. Now, $$\beta^R_{0}(M)
=\beta^R_{0}(\Sigma^{-n}X)=\beta^R_{n}(X)=1.$$ Thus $M$ is cyclic, and so $M\cong R_X$. Therefore,  $X\simeq \Sigma^nR_X$,
as required. Now, the remaining assertions are trivial.
\end{proof}

\begin{lemma}\label{27} Let $(R,\fm,k)$ be a local ring. Let $X\in \mathcal{D}^f_{\Box}(R)$ be a non-exact complex with
$\beta^R_{\sup X}(X)=1$. Suppose that $\dim_RX=0$ and $\dim_R(\text{H}_{\sup X}(X))=\sup X$. Then $X\simeq \Sigma^{\sup X}
R/\Ann_RX$.
\end{lemma}

\begin{proof} In view of Lemma \ref{26}, the assertion is equivalent to $\amp X=0$. Let $n:=\sup X$ and $t:=\inf X$. We
claim that $n=t$. On the contrary, assume that $n>t$. Let $$F=\cdots\lo F_i\lo F_{i-1}\lo \cdots \lo F_t\lo 0$$ be the minimal
free resolution of $X$. Then $F_i$ is a free $R$-module, of finite rank $\beta^R_i(X)$, for all integers $i$. From the
definition of dimension of a complex, we get $$0=\dim_RX\geq \dim_R(\text{H}_{t}(X))-t,$$ and so $t\geq 0$.

Our assumptions implies that $\dim_R(\text{H}_i(X))\leq i$ for all $i<n$ and $\dim_R(\text{H}_n(X))=n$. In particular, $n\leq
\dim R$. Thus, by Lemma \ref{22}, we have $\gamma_n(F)>0$ and $\gamma_{n-1}(F)>0$. This yields that $$\beta^R_n(X)=\gamma_n(F)
+\gamma_{n-1}(F)\geq 2.$$ We arrive at a contradiction.
\end{proof}

For every semidualizing complex $C$ for a local ring $R$, Christensen \cite[Corollary 3.4]{C1} has proved the inequality
$\cmd R\leq\cmd_RC+\amp C$. Part (c) of the next result shows that it is an equality, provided $\Ass R$ is a singleton.

\begin{lemma}\label{28} Let $(R,\fm,k)$ be a local ring and $C$ a semidualizing complex.
\begin{itemize}
\item[(a)] $\Ass_R(\text{H}_{\sup C}(C))=\{\fp\in \Ass R \mid \inf C_{\fp}=\sup C\}$.
\item[(b)] $\Assh_RC\subseteq \Ass R$.
\item[(c)] If $\Ass R$ is a singleton, then $\dim_RC=\dim R-\sup C$ and $\cmd R=\cmd_RC+\amp C$.
\end{itemize}
\end{lemma}

\begin{proof} Set $n:=\sup C$ and $d:=\dim_RC$.

(a) By \cite[Corollary A.7]{C1}, a prime ideal $\fp$ of $R$ belongs to $\fp\in \Ass_R(\text{H}_{n}(C))$ if and only
$\depth R_{\fp}=0$ and $\inf C_{\fp}=n$. But, for a prime ideal $\fp$ of $R$, it is evident that $\fp\in \Ass R$
if and only if $\depth R_{\fp}=0$. Thus, $$\Ass_R(\text{H}_{n}(C))=\{\fp\in \Ass R \mid \inf C_{\fp}=n\}.$$

(b) Let $A$ be an Artinian $R$-module. Then $A$ can be given a natural structure as an Artinian $\widehat{R}$-module, and
so $$\Att_RA=\{\fp\cap R\mid \fp\in \Att_{\widehat{R}}A\}.$$  Also, we know that there is a natural $\widehat{R}$-isomorphism  $A\otimes_R\widehat{R}\cong A$. By \cite[Proposition 2.1]{HD}, the $R$-module $\text{H}_{\fm}^d(C)$ is Artinian. This and
\cite[Corollary 3.4.4]{Li} yield, respectively, the following $\widehat{R}$-isomorphisms: $$\text{H}_{\fm}^d(C)\cong
\text{H}_{\fm}^d(C)\otimes_R\widehat{R}\cong \text{H}_{\widehat{\fm}}^d(C\otimes_R\widehat{R}).$$ By \cite[Lemma 2.6]{C1},
it follows that $C\otimes_R\widehat{R}$ is a semidualizing complex for the local ring $\widehat{R}$ and we may check that $\dim_{\widehat{R}}(C\otimes_R\widehat{R})=\dim_RC$. Thus, by \cite[Lemma 2.5]{HD}, we deduce that
$$\begin{array}{ll}
\Assh_RC&=\Att_R(\text{H}_{\fm}^d(C))\\
&=\{\fp\cap R\mid \fp\in \Att_{\widehat{R}}(\text{H}_{\widehat{\fm}}^d(C\otimes_R\widehat{R}))\}\\
&=\{\fp\cap R\mid \fp\in \Assh_{\widehat{R}}(C\otimes_R\widehat{R})\}.
\end{array}$$
On the other hand, it is known that $$\Ass R=\{\fp\cap R\mid \fp\in \Ass \widehat{R}\}.$$ Therefore, we may and do assume that $R$
is complete. Then $R$ possesses a normalized dualizing complex $D$. Now, the Local Duality theorem \cite[Chapter V, Theorem 6.2]{H}
implies that $$\text{H}_{\fm}^d\left(C\right)\cong \Hom_R\left(\text{H}_d(C^\dag),\text{E}_R(R/\fm)\right),$$ and so $$\Assh_RC=\Att_R(\text{H}_{\fm}^d(C))=\Ass_R(\text{H}_d(C^\dag)).$$ By parts (f) and (b) of Remark \ref{24}, $C^\dag$ is a semidualizing
complex for $R$ with $\sup C^\dag=d$. Thus, (a) yields that $\Assh_RC\subseteq \Ass R$.

(c) Let $\fp$ be the unique member of $\Ass R$. Then, by virtue of (b) and (a), we have $$\Assh_RC=\{\fp\}=
\Ass_R(\text{H}_{n}(C)).$$ Hence, $$\dim_RC=\dim R/\fp-\inf C_{\fp}=\dim R-n.$$ By \cite[Corollary 3.2]{C1}, it turns out
that $$\depth R=\depth_RC+\inf C.$$ This implies that
$$\begin{array}{ll}
\cmd_RC+\amp C&=\dim R-n-\depth_RC+n-\inf C\\
&=\dim R-\depth R\\
&=\cmd R.
\end{array}$$
\end{proof}

Now, we are ready to prove our first theorem.

\begin{theorem}\label{29} Let $(R,\fm,k)$ be a local ring. Let $X\in \mathcal{D}^f_{\Box}(R)$ be a non-exact complex with
$\beta^R_{\sup X}(X)=1$.
\begin{itemize}
\item[(a)] Assume that $X$ is Cohen-Macaulay. Then $X\simeq \Sigma^{\sup X}R/\Ann_RX$.
\item[(b)] Assume that $\Ass R$ is a singleton and $X$ is a semidualizing complex for $R$. Then $X\simeq \Sigma^{\sup X}R$.
\end{itemize}
\end{theorem}

\begin{proof} (a) By replacing $X$ with $\Sigma^{\dim_RX}X$, we may further assume that $\dim_RX=0$. Then, by Lemma \ref{25},
we have $\dim_R(\text{H}_{\sup X}(X))=\sup X$. Now, Lemma \ref{27} yields the claim.

(b) By replacing $X$ with $\Sigma^{\dim R-\sup X}X$, we may and do assume that $\sup X=\dim R$. Then, Lemma \ref{28}(c) implies
that $\dim_RX=0$. On the other hand, by Lemma \ref{28}(a), we deduce that $$\dim_R(\text{H}_{\sup X}(X))=\dim R=\sup X.$$ Now,
Lemma \ref{27} yields that $X\simeq \Sigma^{\sup X}R/\Ann_RX$. But, then $R/\Ann_RX$ would be a semidualizing module for $R$,
which in turns implies that $\Ann_RX=0$.
\end{proof}

The next result provides an affirmative answer to a conjecture raised by Foxby \cite[Conjecture B]{F1}.

\begin{corollary}\label{210} Let $(R,\fm,k)$ be a local ring. Let $M$ be a non-zero finitely generated $R$-module of type one.
Then $M$ is the dualizing module for the local ring $R/\text{Ann}_RM$. In particular, the local ring $R/\Ann_RM$ is Cohen-Macaulay.
\end{corollary}

\begin{proof}  Without loss of generality, we may and do assume that $R$ is complete. Then $R$ possesses a normalized dualizing
complex $D$. Set $X:=M^{\dag}$ and $n:=\dim_RM$. By using parts (a),(b) and (c) of Remark \ref{24}, we can see that $X$ is a
Cohen-Macaulay complex and $\sup X=n$. By Remark \ref{24}(d), we conclude that $\beta^R_n(X)=1$, and so Theorem \ref{29}(a)
yields that $X\simeq \Sigma^n R/\Ann_RM$. Note that Lemma \ref{25}(c) implies that $\Ann_RX=\Ann_RM$.

As $$(R/\Ann_RM)^{\dag}={\bf R}\Hom_R(R/\Ann_RM,D)$$ is a dualizing complex for the local ring $R/\Ann_RM$ and $$M\simeq (M^{\dag})
^{\dag}\simeq X^{\dag}\simeq \Sigma^{-n}(R/\Ann_RM)^{\dag},$$ it follows that $M$ is the dualizing module for the local ring
$R/\text{Ann}_RM$.
\end{proof}

Although the next result is known to the experts, we give a proof for the sake of completeness.

\begin{lemma}\label{211} Let $(R,\fm,k)$ be a local ring. Let $$F=\cdots\lo F_i\overset{\partial^F_i}\lo F_{i-1}\lo \cdots
\overset{\partial^F_{t+1}}\lo F_t\lo 0$$ be a complex of finitely generated free $R$-modules such that $\im  \partial^F_i
\subseteq \fm F_{i-1}$ for all integers $i>t$.
\begin{enumerate}
\item[(a)] If $F\simeq 0$, then $F_i=0$ for all $i\geq t$.
\item[(b)] Assume that $F$ is indecomposable in  $\mathcal{D}(R)$ and $\partial^F_j$ is zero for some $j>t$.
Then $F_i=0$ either for all $i\geq j$ or for all $i\leq j-1$.
\end{enumerate}
\end{lemma}

\begin{proof} (a) We can split $F$ into the short exact sequences
$$0\longrightarrow \im \partial ^{F}_{t+2}\hooklongrightarrow F_{t+1}\longrightarrow F_t\longrightarrow 0$$ $$\vdots$$
$$0\longrightarrow \im \partial ^{F}_{i+2}\hooklongrightarrow F_{i+1}\longrightarrow  \im \partial ^{F}_{i+1}
\longrightarrow 0$$ $$\vdots$$ Then, it turns out that the $R$-module $\im \partial^{F}_{i}$ is projective for all $i>t+1$
and all of the above short exact sequences are split. Thus $F$ is split-exact, and so the complex $k\otimes _RF$ is split-exact
too. Since, $\im \partial^F_{i}\subseteq \fm F_{i-1}$ for all $i>t$, it follows that all differential maps of the
complex $k\otimes_R F$ are zero. Hence, $$0=\operatorname{H}_i (k\otimes_R F)=k\otimes_R F_{i}$$ for all $i\geq t$.
Now, by Nakayama's lemma, $F_i=0$ for all $i\geq t$.

(b) Set $$Y :=\dots \longrightarrow F_l \longrightarrow F_{l-1} \longrightarrow \dots \longrightarrow F_j\longrightarrow 0$$
and $$Z :=0\longrightarrow F_{j-1}\longrightarrow F_{j-2}\longrightarrow \dots \longrightarrow F_t\longrightarrow 0.$$ Then,
clearly we have $F\cong Y\oplus Z $. So, either $Y\simeq 0$ or $Z\simeq 0$. Now, the claim follows by $(a)$.
\end{proof}

Next, we prove our second theorem.

\begin{theorem}\label{212} Let $(R,\fm,k)$ be a local domain and $C$ a semidualizing complex. If $\beta^R_{\sup C}(C)=2$, then
$\amp C=0$.
\end{theorem}

\begin{proof} Set $n:=\sup C$. If $\pd_RC<\infty$, then \cite[Proposition 8.3]{C1} implies that $C\simeq R$, and so $\amp C=0$.
So, in the rest of the proof, we assume that $\pd_RC=\infty$.

Let $$F=\cdots\lo F_i\lo F_{i-1}\lo \cdots \lo F_t\lo 0$$ be the minimal free resolution of $C$. Then $F_i$ is a free $R$-module,
of finite rank $\beta^R_i(C)$, for all $i\geq t$. Since, $R$ is indecomposable and $R\simeq {\bf R}\Hom_R(C,C)$, it follows
that $C$ and, consequently, $F$ are indecomposable in $\mathcal{D}(R)$. Since $F_n\neq 0$, in view of Lemma \ref{211}(b), vanishing
of the map $\partial^F_{n+1}$ implies that $F_i=0$ for all $i>n$. So, as $\pd_RC=\infty$, we deduce that $\partial^F_{n+1}\neq 0$.

Next, we show that $\partial^F_{n}=0$. As $\im \partial^F_{n}$ is torsion-free, we equivalently prove that $\rank_R(\im \partial^F_{n})
=0$. Lemma \ref{28}(a) yields that $\Ass_R(\text{H}_{n}(C))\subseteq \Ass R$. Hence $\Ass_R(\text{H}_{n}(C))=\{0\}$, and so
$\text{H}_{n}(C)$ is torsion-free. As the $R$-module $\text{H}_{n}(C)$ is non-zero and torsion-free, it follows that
$\rank_R(\text{H}_{n}(C))\geq 1$. As, also, $\rank_R(\im \partial^F_{n+1})$ is positive, from the short exact sequence $$0\lo \im \partial^F_{n+1}\lo \ker \partial^F_{n}\lo \text{H}_n(F)\lo 0,$$ we get that $\rank_R(\ker \partial^F_{n})\geq 2$. Finally, in view
of the short exact sequence $$0\lo \ker \partial^F_{n}\lo F_n \lo \im \partial^F_{n} \lo 0,$$ we see that
$\rank_R(\im \partial^F_{n})=0$.

Now as $\partial^F_{n}=0$ and $F_n\neq 0$,  Lemma \ref{211}(b) yields that $F_i=0$ for all $i\leq n-1$. In particular,
$\inf C=\inf F=n$, and so $\amp C=0$.
\end{proof}

Each of parts (a) and (b) of the next result generalizes the main result of \cite{R}. Also, its part (c)
extends the main result of \cite{CHM}.

\begin{corollary}\label{213} Let $(R,\fm,k)$ be a local ring and $C$ a semidualizing complex for $R$.
\begin{enumerate}
\item[(a)] Assume that $C$ is a module of type one. Then $C$ is dualizing.
\item[(b)] Assume that $C$ has type one. Then $C$ is dualizing; provided either\\
$(1)$ $\Ass R$ is a singleton and $R$ admits a dualizing complex; or\\
$(2)$ $\Ass \widehat{R}$ is a singleton.
\item[(c)] Assume that $R$ is a domain admitting a dualizing complex and $C$ has type two. Then $C$ is Cohen-Macaulay.
\end{enumerate}
\end{corollary}

\begin{proof} (a) is immediate by  Corollary \ref{210}. Note that $\Ann_RC=\Ann_RR=0$.

By \cite[Lemma 2.6]{C1}, $C\otimes_R\widehat{R}$ is a semidualizing complex for $\widehat{R}$. It is easy to verify that the
$R$-complex $C$ and the $\widehat{R}$-complex $C\otimes_R\widehat{R}$ have the same type. On the other hand, by
\cite[Proposition 5.5]{AF}, we have
$$\begin{array}{ll}
\id_RC&=-\inf ({\bf R}\Hom_R(k,C))\\
&=-\inf ({\bf R}\Hom_R(k,C)\otimes_R\widehat{R})\\
&=-\inf ({\bf R}\Hom_{\widehat{R}}(k,C\otimes_R\widehat{R}))\\
&=\id_{\widehat{R}}(C\otimes_R\widehat{R}).
\end{array}
$$
So in part (2) of (b), we may also assume that $R$ possesses a dualizing complex. Let $D$ be a normalized dualizing complex for $R$. Set $n:=\dim_RC$. By Remark \ref{24}(f), $L:=C^{\dag}$ is also a semidualizing complex for $R$. By parts (b) and (d)
of Remark \ref{24}, we have $\sup L=n$ and $\beta^R_n(L)=\mu_R^n(C)$.

(b) We have $\beta^R_n(L)=1$ and $\Ass R$ is a singleton. Hence, Theorem \ref{29}(b) implies that $L\simeq \Sigma^{n}R$. So, by
Remark \ref{24}(a), we have $C\simeq \Sigma^{-n}D$.

(c) We have $\beta^R_n(L)=2$. Theorem \ref{212} yields that $\amp L=0$, and so, in view of Remark \ref{24}(c), we conclude that
$C$ is Cohen-Macaulay.
\end{proof}

The next result is a useful tool for reducing the study of Betti numbers of semidualizing modules over Cohen-Macaulay
local rings to the same study on Artinian local rings.

\begin{lemma}\label{214} Let $(R,\fm,k)$ be a $d$-dimensional Cohen-Macaulay local ring and $C$ a semidualizing
module for $R$. Let $\underline{x}=x_1, x_2, \dots, x_d\in \fm$ be an $R$-regular sequence and set $\overline{R}:
=R/\langle \underline{x}\rangle$ and $\overline{C}:=C/\langle \underline{x}\rangle C$. Then $\overline{C}$ is a
semidualizing module for the Artinian local ring $\overline{R}$ and $\beta^{\overline{R}}_{i}(\overline{C})=
\beta^R_{i}(C)$ for all integers $i\geq 0$. Furthermore, $C\cong R$ if and only if $\overline{C}\cong \overline{R}$.
\end{lemma}

\begin{proof} By \cite[Theorem 2.2.6]{S}, $\overline{C}$ is a semidualizing module for the Artinian local ring $\overline{R}$.
Let $F$ be a free resolution of the $R$-module $C$. By \cite[Proposition 1.1.5]{BH}, we conclude that $\overline{R}\otimes_RF$ is
a free resolution of the $\overline{R}$-module $\overline{C}$. Hence, for every non-negative integer $i$, we have the following
natural $R$-isomorphisms: $$\Tor^{\overline{R}}_i(k,\overline{C})\cong \text{H}_i(k\otimes_{\overline{R}}(\overline{R}\otimes_RF))
\cong \text{H}_i(k\otimes_RF)\cong \Tor^{R}_i(k,C).$$
Thus $\beta^{\overline{R}}_{i}(\overline{C})=\beta^R_{i}(C)$ for all $i\geq 0$. The remaining assertion is immediate by
\cite[Proposition 4.2.18]{S}.
\end{proof}

We apply the following obvious result in the proof of Theorem \ref{216}.

\begin{lemma}\label{215} Let $K$ be an $R$-module and $0\lo L\lo M\lo N\lo 0$ an short exact sequence of $R$-homomorphisms such that
$\Ext_R^i(M,K)=0=\Ext_R^i(N,K)$ for all $i>0$. Then the sequence $$0\lo \Hom_R(N,K)\lo \Hom_R(M,K)\lo \Hom_R(L,K)\lo 0$$
is exact and $\Ext_R^i(L,K)=0$ for all $i>0$.
\end{lemma}

The following result is also needed in the proof of Theorem \ref{216}. It is interesting in its own right.

\begin{lemma}\label{2151} Let $(R,\fm,k)$ be a local ring and $M$ a finitely generated $R$-module with a rank. If
the projective dimension of $M$ is infinite, then $\beta^R_{n}(M)\geq 2$ for all $n\in \mathbb{N}_0$.
\end{lemma}

\begin{proof} On the contrary, assume that $\beta^R_{n}(M)<2$ for some $n\in \mathbb{N}_0$. Then $\beta^R_{n}(M)=1$, because
$\pd_RM=\infty$. For each non-negative integer $i$, set $\beta_i:=\beta_i^R(M)$. Hence, there is an exact sequence $$F=\cdots
\lo R^{\beta_i}\overset{\partial_i}\lo R^{\beta_{i-1}}\lo \cdots \overset{\partial_{1}}\lo R^{\beta_0}\overset{\partial_{0}}
\lo M\lo 0.$$ For each natural integer $i$, set $\Omega^iM:=\ker \partial_{i-1}$. Splitting $F$ into short exact sequences and
applying \cite[Proposition 1.4.5]{BH}, successively, to them, yields that each $\Omega^iM$ has a rank. Since $\Omega^iM$ is
torsion-free and non-zero, it follows that $\rank_R (\Omega^{i}M)>0$ for all $i\in \mathbb{N}$. By applying \cite[Proposition
1.4.5]{BH} to the short exact sequence $$0\lo \Omega^{n+1}M \lo R\lo \Omega^nM \lo 0,$$ it turns out that $$\rank_R (\Omega^{n+1}M)+
\rank_R (\Omega^{n}M)=1.$$ Hence, either $\rank_R (\Omega^{n}M)=0$ or $\rank_R (\Omega^{n+1}M)=0$, and we arrive at the desired
contradiction.
\end{proof}

Finally, we are in a position to prove our third (and last) theorem.

\begin{theorem}\label{216} Let $(R,\fm,k)$ be a local ring and $C$ a semidualizing complex for $R$ such that $\beta^R_{n}(C)=1$
for some integer $n$. Assume that either $R$ is Cohen-Macaulay or $C$ is a module with a rank. Then $C\simeq \Sigma^nR$ and $n=\sup C$.
\end{theorem}

\begin{proof} First, we assume that $R$ is Cohen-Macaulay. By \cite[Corollary 3.4(a)]{C1}, we have $\amp C\leq \cmd R$, and so we
may and do assume that $C$ is a module. Then, it would be enough to show that $C\cong R$. On the contrary, suppose that $C\ncong R$.
By Lemma \ref{214}, we may and do assume that $R$ is Artinian. By \cite[Corollaries 2.1.14 and 2.2.8]{S}, we know that if $C$ is
cyclic or $\pd_RC<\infty$, then $C\cong R$. So, we may and do assume that $n\geq 1$ and $\pd_RC=\infty$, and look for a contradiction. For
each non-negative integer $i$, set $\beta_i:=\beta_i^R(C)$. Hence, the minimal free resolution of $C$ has the form
$$\cdots \lo R^{\beta_{n+1}}\lo R\lo R^{\beta_{n-1}}\lo \cdots \lo R^{\beta_{1}}\lo R^{\beta_{0}}\lo 0.$$ Then, we have
the following short exact sequences:
$$(0) \ \  \  \  \  \  \ \    \  \  \ \   \  0\lo \Omega^1C\lo R^{\beta_{0}}\lo C\lo 0$$
$$(1)  \ \  \ \  \   \  \ \ \  \  \  \   \  \  0\lo \Omega^2C\lo R^{\beta_{1}}\lo \Omega^1C\lo 0$$
$$ \  \ \  \   \  \ \ \   \      \   \  \  \  \     \vdots$$
$$(n-1)  \  \   \ \  \  \  \ \   \   \  \  \   \  0\lo \Omega^nC\lo R^{\beta_{n-1}}\lo \Omega^{n-1}C\lo 0$$
$$(n)  \   \  \  \  \  \ \  \  \  \  \  \  \  0\lo \Omega^{n+1}C\lo R\lo \Omega^nC\lo 0.$$
We successively apply the contravariant functor $(-)^*:=\Hom_R(-,C)$ to the short exact sequences $(0), (1), \dots, (n)$ and use
Lemma \ref{215} to get the following short exact sequences:
$$(0)^* \ \  \  \  \  \  \ \    \  \  \ \   \  0\lo R\lo C^{\beta_{0}}\lo (\Omega^1C)^*\lo 0$$
$$(1)^*  \ \  \ \  \   \  \ \ \  \  \  \   \  \  0\lo (\Omega^1C)^*\lo C^{\beta_{1}}\lo (\Omega^2C)^*\lo 0$$
$$ \  \ \  \   \  \ \ \   \      \   \  \  \  \     \vdots$$
$$(n-1)^*  \  \   \ \  \  \  \ \   \   \  \  \   \  0\lo (\Omega^{n-1}C)^*\lo C^{\beta_{n-1}}\lo (\Omega^{n}C)^*\lo 0$$
$$(n)^*  \   \  \  \  \  \ \  \  \  \  \  \  \  0\lo (\Omega^{n}C)^*\lo C\lo (\Omega^{n+1}C)^*\lo 0.$$

As $\pd_RC=\infty$, it follows that $\Omega^iC\neq 0$ for all integers $i\geq 0$. Suppose that $(\Omega^jC)^*=0$ for some
integer $j\geq 0$. Then, by \cite[Proposition 3.1.6]{S}, $(\Omega^{j-1}C)^*\cong C^{\beta_{j-1}}\in \mathcal{B}_C(R)$.
Applying \cite[Proposition 3.1.7(b)]{S}, successively, to the short exact sequences $(j-2)^*, \dots, (1)^*, (0)^*$ yields
that $R\in \mathcal{B}_C(R)$. Now, \cite[Corollary 4.1.7]{S} implies that $C\cong R$. Thus $(\Omega^iC)^*\neq 0$ for all
integers $i\geq 0$.

Set $h:=\beta_{0}-\beta_{1}+\cdots +(-1)^n\beta_{n}$. Using the short exact sequences $(0), (1), \dots, (n)$, we deduce that
$$(A_1): \  \  \  \  \ \ell_R(C)=h\ell_R(R)+(-1)^{n+1}\ell_R(\Omega^{n+1}C)$$ and $$(B_1): \ \ \ \   \  \ \  \  \  \ell_R(\Omega^{n+1}C)<\ell_R(R).$$
Similarly, by using the short exact sequences $(0)^*, (1)^*, \dots, (n)^*$, we may see that
$$(A_2): \  \  \  \  \ \ell_R(R)=h\ell_R(C)+(-1)^{n+1}\ell_R((\Omega^{n+1}C)^*)$$ and $$(B_2):  \  \ \ \   \ \   \  \  \ell_R((\Omega^{n+1}C)^*)<\ell_R(C).$$

Assume that $n$ is even. If $h\leq 0$, then $A_1$ yields that $\ell_R(C)<0$. If $h=1$, then $A_1$ and $A_2$, respectively,
imply that $\ell_R(C)<\ell_R(R)$ and  $\ell_R(R)<\ell_R(C)$. If $h\geq 2$, then $A_1$ and $B_1$ imply that $\ell_R(C)>
\ell_R(R)$. On the other hand,  $A_2$ and $B_2$ imply that $\ell_R(R)>\ell_R(C)$.

Next, assume that $n$ is odd. If $h<0$, then, by $A_1$ and $B_1$, we deduce that $\ell_R(C)<0$. If $h=0$, then $A_1$
and $B_1$ imply that $\ell_R(C)<\ell_R(R)$ and $A_2$ and $B_2$ imply that $\ell_R(R)<\ell_R(C)$. If $h\geq 1$, then $A_1$
and $A_2$, respectively, imply that $\ell_R(C)>\ell_R(R)$ and $\ell_R(R)>\ell_R(C)$.

Therefore if either $n$ is even or odd, we arrive at a contradiction.

Now, assume that $C$ is a module with a rank. Then, Lemma \ref{2151} implies that $C$ has finite projective dimension, and
so $C\cong R$ by \cite[Corollary 2.2.8]{S}.
\end{proof}

The next result extends \cite[Characterization 1.6]{CSV}.

\begin{corollary}\label{217} Let $(R,\fm,k)$ be a local ring and $C$ a semidualizing complex for $R$ such that $\mu_R^{n}(C)=1$
for some integer $n$. Assume that either
\begin{enumerate}
\item[(a)] $R$ is Cohen-Macaulay, or
\item[(b)]$R$ is analytically irreducible and $C$ is Cohen-Macaulay.
\end{enumerate}
Then $C$ is dualizing and $n=\dim_RC$.
\end{corollary}

\begin{proof} We may and do assume that $R$ is complete. Then $R$ possesses a normalized dualizing complex $D$. By parts (f) and
(d) of Remark \ref{24}, it turns out that $L:=C^{\dag}$ is also a semidualizing complex for $R$ and $\beta^R_n(L)=1$. Set $d:=
\dim_RC$. Then,  by Remark \ref{24}(b), $d=\sup L$.

(a) Assume that $R$ is Cohen-Macaulay. Then Theorem \ref{216} implies that $L\simeq \Sigma^{n}R$ and $n=d$. Thus, by Remark
\ref{24}(a), we deduce that $C\simeq \Sigma^{-d}D$.

(b)  Assume that $R$ is analytically irreducible and $C$ is Cohen-Macaulay. As $C$ is Cohen-Macaulay, Remark \ref{24}(c) implies
that $\amp L=0$. Since $R$ is a domain, it follows that $\text{H}_{d}(L)$ is a semidualizing module with a rank, and so Theorem
\ref{216} implies that $$L\simeq \Sigma^{d}\text{H}_{d}(L)\simeq \Sigma^{d}R.$$ Using this, as in the proof (a), we can obtain
$C\simeq \Sigma^{-d}D$.
\end{proof}

%%%%%%%%%%%%%%%%%%%%%%%%%%%%%%%%%%%%%%%%%%%%%%%%%%%

\end{document}